\documentclass[final,leqno,onefignum]{amsart}

\usepackage{fancyhdr}
\usepackage{t1enc,amsfonts,latexsym,amsmath,verbatim,amssymb,epsfig,graphics}

\usepackage{epsfig,amsfonts,latexsym,amsmath,amssymb}
\usepackage{t1enc,verbatim,graphics}

\author[F. Andersson]{Fredrik Andersson}\address{Centre for Mathematical Sciences, Lund University, Sweden}
\email{fa@maths.lth.se}
\author[M. Carlsson]{Marcus Carlsson}\address{Centre for Mathematical Sciences, Lund University, Sweden}
\email{marcus.carlsson@math.lu.se}
\author[K.-M. Perfekt]{Karl-Mikael Perfekt}\address{Department of Mathematical Sciences, Norwegian University of Science and Technology, NO-7491 Trondheim, Norway}
\email{karl-mikael.perfekt@math.ntnu.no}

\newcommand{\R}{\mathbb{R}}
\newcommand{\re}{\mathsf{Re}}

\newtheorem{theorem}{Theorem}[section]
\newtheorem{proposition}[theorem]{Proposition}
\newtheorem{lemma}[theorem]{Lemma}
\newtheorem{corollary}[theorem]{Corollary}

\def\F {f}
\def\C {\mathbb{C}}

\DeclareMathOperator{\Lip}{Lip}

\title{Operator-Lipschitz estimates for the singular value functional calculus}

\begin{document}
\maketitle

\begin{abstract}
We consider a functional calculus for compact operators, acting on the singular values rather than the spectrum, which appears frequently in applied mathematics. Necessary and sufficient conditions for this singular value functional calculus to be Lipschitz-continuous with respect to the Hilbert-Schmidt norm are given. We also provide sharp constants.
\end{abstract}
\section{Introduction}
For simplicity we restrict attention to the finite-dimensional matrix case in this introduction. Let $A$ be a matrix with singular value decomposition $A=U\Sigma V^*$, and consider the operation of changing the singular values by applying some function $\F:\mathbb{R}_+\rightarrow \C$ to $\Sigma$, thus yielding a new matrix which we will call $\F_s(A)$, where the subscript $s$ indicates that we are considering a \textit{singular value functional calculus.} For matrices with non-trivial nullspaces, it is easy to see that the condition $\F(0)=0$ is necessary for $\F_s(A)$ to be well defined (Section \ref{s2}). Let us also remark that, in case $\F$ is a function defined on $\mathbb{C}$ and $A$ is a normal matrix, then $\F(A)$ is defined by the classical functional calculus (CFC) based on the spectral theorem. However, it is rarely the case that $\F(A)=\F_s(A)$ except when $A$ is positive (Section \ref{s2}).

The operation $A\mapsto \F_s(A)$ is commonly seen in applied mathematics, since it often appears as the proximal operator \cite{rw} in Matrix Optimization Problems and Compressed Sensing. For applications in Computer Vision, Structure from Motion, Photometric Stereo and Optical Flow, see \cite{larsson2,larsson1} and the references therein. See \cite{chu} for its use in alternating projection schemes and \cite{anderssonalternating} for a problem in financial mathematics \cite{higham}. For applications in Control Systems see \cite{fazel}, MRI see \cite{candesMRI}, and for applications to complex frequency estimation see \cite{actwIEEE}. More examples can be found in \cite{ding, recht}.

When studying convergence of algorithms utilizing the singular value functional calculus, it is important to have bounds for the distance $\|\F_s(A)-\F_s(B)\|_{F}$ given $\|A-B\|_{F}$, where the subscript $F$ indicates that we are dealing with the Frobenius norm, (which is the same as the Hilbert-Schmidt norm $\|\cdot\|_{\mathcal{S}_2}$, but we will follow standard conventions and use this notation only for the infinite-dimensional case). We thus define \begin{equation}\label{SiOpLip}\|\F_s\|_{\Lip}=\sup_{A,B}\frac{\|\F_s(A)-\F_s(B)\|_{F}}{\|A-B\|_{F}}.\end{equation} In Section \ref{s4} we shall show that this supremum turns out not to depend on the dimension of the matrices $A,B$, which is why this is omitted from the notation.

If one restricts attention to positive matrices in the above supremum, then it is known that it equals $\|\F\|_{\Lip}$. This follows from more general work concerning the CFC-case \cite{kittaneh}, a result which is also presented with a very simple proof in \cite[Lemma VII.5.4]{bhatia} and was rediscovered in \cite{wihler}. We remark that in the CFC-case, results concerning H\"older and Lipschitz continuity with respect to various operator norm have a long history, see e.g. \cite{potapov,farforovskaya,aleksandrov} and the references therein.

The main result of this paper is the following:
\begin{theorem}\label{aoa}
Let $\F:\mathbb{\R}_+\rightarrow \C$ be continuous with $\F(0)=0$. Then $\|\F_s\|_{\Lip}\leq \sqrt{2}\|\F\|_{\Lip}$ and the constant $\sqrt{2}$ is the best possible (if it is to hold for all Lipschitz functions $\F$). However, if $\F$ is real valued, then $$\|\F_s\|_{\Lip}= \|\F\|_{\Lip}.$$
\end{theorem}
For an interesting related result also having $\sqrt{2}$ as the best constant in the complex case and $1$ in the real case, see \cite{araki}. In terms of applications, the study of real valued functions is most relevant. Based on general arguments \cite{ding}, one can show that $\|\F_s\|_{\Lip}$ should be bounded in terms of $\|\F\|_{\Lip}$, but the fact that the constant is 1 is rather surprising and is likely to have an impact on algorithmic design involving the singular value functional calculus.

\section{The singular value functional calculus}\label{s2}

Let $\mathcal{H}$ be a separable Hilbert space of dimension $d$, $1\leq d\leq \infty$, and suppose that
$A : \mathcal{H} \to \mathcal{H}$ is a compact operator, $A \in \mathcal{B}_0$. Then
$A$ has singular value decomposition; there exist
an orthonormal basis $(u_n)_{n=0}^{d}$ of $\mathcal{H}$ and
an orthonormal sequence $(v_n)_{n=0}^{d}$ such that
\begin{equation} \label{eq:svddecomp}
 A = \sum_{n=0}^{d} s_n(A) \, u_n \otimes v_n,
\end{equation}
where $s_n(A)$ are the singular values of $A$. In other words $$A h = \sum_{n=0}^d s_n(A) \langle h, v_n \rangle u_n. $$
An equivalent formulation of \eqref{eq:svddecomp} is that
$A$ has a polar decomposition $A = U \Sigma$,
where $U$ is a partial isometry and $\Sigma = |A|$ is a positive
diagonalizable operator such that $\Sigma v_n = s_n(A) v_n$.
We will primarily be concerned with operators $A$ in the Hilbert-Schmidt class $\mathcal{S}_2$, i.e. compact operators such that $$\|A\|_{\mathcal{S}_2}^2=\sum_{n=0}^d s_n^2(A)<\infty.$$
Following standard conventions we denote this norm by $\|A\|_F^2$ whenever $d<\infty$, in which case it coincides with the $l^2$-norm of the elements of the matrix representation of $A$ in any orthonormal basis.

Given any continuous function $\F:\mathbb{R}_+\rightarrow\mathbb{C}$ such that $\F(0) = 0$ we define $\F_s: \mathcal{B}_0 \to \mathcal{B}_0$ by
\begin{equation} \label{eq:funcdef}
 \F_s(A) = \sum_{n=1}^\infty \F(s_n(A)) \, u_n \otimes v_n,
\end{equation}
or equivalently that $\F_s(A) = U \F(\Sigma)$,
 where $\F(\Sigma)$ is the operator such that $\F(\Sigma) v_n = \F(s_n(A)) v_n$. The
subscript $s$ indicates that we are dealing with a ``singular value
functional calculus''. To see that it is well defined, note that if $s=s_n(A)\neq 0$ and $h\in \ker (s^2 I-A^*A)$, then $$\F_s(A)h=\frac{\F(s)}{s}Ah.$$
However, note that if $s=0$ and we were to allow $\F(0)\neq 0$ it is clear that $\F_s(A) h$ could depend on the particular choice of  $(u_n)_n$ and $(v_n)_n$.

We remark that $\F$ only needs to be defined on $\mathbb{R}^+$ for $\F_s(A)$
to exist, and obviously $$\F_s(A)=\F(A)$$ for all positive operators $A$. For normal operators, the situation is more complex. Consider for example
\begin{equation*}A=\left(
              \begin{array}{cc}
                0 & 1 \\
                1 & 0 \\
              \end{array}
            \right)=\left(
              \begin{array}{cc}
                \tiny{\frac{1}{\sqrt{2}}} & \tiny{\frac{1}{\sqrt{2}}} \\
                \tiny{\frac{1}{\sqrt{2}}} & -\tiny{\frac{1}{\sqrt{2}}} \\
              \end{array}
            \right)\left(
              \begin{array}{cc}
                1 & 0 \\
                0 & -1 \\
              \end{array}
            \right)\left(
              \begin{array}{cc}
                \tiny{\frac{1}{\sqrt{2}}} & \tiny{\frac{1}{\sqrt{2}}} \\
                \tiny{\frac{1}{\sqrt{2}}} & -\tiny{\frac{1}{\sqrt{2}}} \\
              \end{array}
            \right),
            \end{equation*}
which has singular value decomposition $U\Sigma V^*=AII$. Then $$\F_s(A)=A\F(I)I=\left(
              \begin{array}{cc}
                0 & \F(1) \\
                \F(1) & 0 \\
              \end{array}
            \right)$$ whereas $\F(A)$ is not even defined in
            an the classical functional calculus, due to the negative eigenvalue $-1$. Moreover, if $\F$ is defined on $\R$ we clearly have $\F(A)\neq \F_s(A)$ unless $\F(1)=-\F(-1)$.
            This is further highlighted by the next proposition.
\begin{proposition}
Let $\F:\C\rightarrow\C$ be a continuous function with $\F(0) = 0$. Then $\F_s(A)=\F(A)$ for a normal compact operator $A : \mathcal{H} \to \mathcal{H}$ if and only if $\F$ satisfies $$\F(\lambda)=\frac{\lambda\F(|\lambda|)}{|\lambda|}$$
for every non-zero eigenvalue $\lambda$ of $A$.
\end{proposition}
\begin{proof}
Since $A$ is a compact normal operator, there is an orthonormal basis $(v_n)_n$ of $\mathcal{H}$ such that
\begin{equation} \label{eq:spectraldecomp}
 A = \sum_{n=1}^\infty \lambda_n \, v_n \otimes v_n,
\end{equation}
where $\lambda_n$ are the eigenvalues of $A$, implicitly omitting the zero eigenvalues from the sum \eqref{eq:spectraldecomp}. On the other hand, an SVD of $A$ is given by
$$ A = \sum_{n=1}^\infty |\lambda_n| \, u_n \otimes v_n,$$
where $u_n = \overline{\lambda_n}/|\lambda_n|v_n$.
Therefore, we obtain that
$$\F(A) = \sum_{n=1}^\infty \F(\lambda_n) \, v_n \otimes v_n,$$
while
$$ \F_s(A) = \sum_{n=1}^\infty \frac{\lambda_n \F(|\lambda_n|)}{|\lambda_n|} \, v_n \otimes v_n,$$
proving the proposition.
\end{proof}

\begin{corollary}
Let $p$ be a polynomial without constant term. Then $p_s(A)=p(A)$ for all normal compact operators if and only if $p(z)=\alpha z$, $\alpha\in\C$.
\end{corollary}

\section{Complex doubly substochastic matrices}\label{s3}
This section contains the main technical tool for the proof of Theorem \ref{aoa}.
We say that a square matrix is complex doubly
substochastic (cdss) if for each row and column, the $\ell^1$-sum of entries is less than or equal to 1. (In \cite{simon}, such matrices are simply called doubly substochastic. However, most other sources using this term include a non-negativity condition on the elements, which is why we have chosen to clarify by adding \textit{complex}.) Our main interest in cdss-matrices stems from the fact that if $U$ and $V$ are unitary matrices then
$U\odot V$ is cdss, where $\odot$ denotes the Hadamard product, as follows immediately by the Cauchy-Schwarz inequality.

Let $\pi$ denote any permutation of length $n$ and let $\gamma$ be a
vector of the same length containing unimodular entries. We
denote by $M_{\pi,\gamma}$ the $n \times n$ matrix whose $(j,\pi_j)$'th value is
$\gamma_j$, all other entries zero.  The following lemma is likely known, but lacking a reference we provide a simple proof based on the Birkhoff-von Neumann theorem, (see e.g. \cite{bhatia}).

\begin{lemma} \label{lem:birk}
An  $n \times n$ matrix is complex doubly substochastic if and only if it lies in the convex hull of
$\{M_{\pi,\gamma}:\pi,\gamma\}$.
\end{lemma}
\begin{proof}
Let $\mathcal{V}$ denote the set of cdss $n \times n$-matrices. It is clearly a closed convex set. We shall show that the extreme points of $\mathcal{V}$ are precisely the matrices of the form $M_{\pi,\gamma}$ for some permutation $\pi$ and vector $\gamma$. The lemma then follows immediately by the Krein-Milman theorem (or rather, Minkowski's theorem on convex sets \cite{minkowski}, since we are in Euclidean space.)

First, let $\pi$ and $\gamma$ be given. Let $m_{ij}$ denote the $ij$:th entry of $M_{\gamma, \pi}$. Suppose that $M_{\gamma, \pi} =(A+B)/2$ for matrices $A, B \in \mathcal{V}$ with entries $a_{ij}$ and $b_{ij}$, respectively. Suppose that $|m_{ik}| = 1$. Since $|a_{ik}| \leq 1$ and $|b_{ik}| \leq 1$ this forces that $a_{ik} = b_{ik} = m_{ik}$ and hence that $a_{ij} = b_{ij} = 0$ for $1 \leq j \leq n$, $j \neq k$. Therefore $A = B = M_{\gamma, \pi}$, and thus $M_{\gamma, \pi}$ is an extreme point of $\mathcal{V}$.

For the converse, let $M \in \mathcal{V}$, with entries $m_{ij}$, be an extreme point. Consider first the case where $M$ has a row with sum of absolute values strictly less than 1, say the $p:$th row. Then clearly
\begin{equation*} \label{eq:rowless1}
\sum_{i=1}^n\sum_{j=1}^n |m_{ij}| < n,
\end{equation*}
which, upon changing the order of summation, shows that there is also a column with a sum of absolute values strictly less than 1, say the $q$:th column. Let $E$ be the matrix with its $pq$:th entry $\varepsilon$, all other entries $0$, and let $A = M - E$ and $B = M + E$. For sufficiently small $\varepsilon$ we find that $A$ and $B$ are cdss, which contradicts the fact that $M$ is an extreme point of $\mathcal{V}$, since $M = (A+B)/2$.

We have thus shown that the extreme point $M$ has sums of absolute values of all rows and colums equal to 1. In other words, the matrix $M^|$ with entries $|m_{ij}|$ is a doubly stochastic matrix. By the Birkhoff-von Neumann theorem, $M^|$ is either of the form $M_{\pi, \gamma}$ for a permutation $\pi$ and $\gamma = (1, 1, \ldots, 1)$ or it is not an extreme point and thus of the form $M^| = (A+B)/2$ for two doubly stochastic matrices $A$ and $B$, $A \neq B$. In the former case, $M$ is  of the form $M_{\pi, \tilde{\gamma}}$ for a suitable sequence $\tilde{\gamma}$. In the latter case, as seen by adjusting the arguments of the entries of $A$ and $B$, $M$ is clearly not an extreme point of $\mathcal{V}$, a contradiction.
\end{proof}
\section{Operator Lipschitz estimates}\label{s4}

Throughout we let $\F:\mathbb{R}_+\rightarrow\mathbb{C}$ be Lipschitz with $\F(0)=0$. Let $\lambda_1, \lambda_2 \in \C$ be two scalars interpreted as operators on $\mathcal{H}=\mathbb{C}$. Then $\F_s(\lambda_j) = \frac{\lambda_j}{|\lambda_j|}\F(|\lambda_j|)$ for $j=1,2$. Hence a Lipschitz condition $\|\F_s(A) - \F_s(B)\|_{\mathcal{S}_2} \leq C \|A - B\|_{\mathcal{S}_2} $ implies that \begin{equation}\label{optimal}|c_1 \F(x) - c_2\F(y)| \leq C |c_1 x - c_2 y|\end{equation} for all $x, y \geq 0$ and $c_1, c_2\in \mathbb{T}$, where $\mathbb{T}$ denotes the unit circle in $\mathbb{C}$. This motivates the following definition
\begin{equation}\label{apa}
\|\F\|_{\Lip-\mathbb{C}}=\sup_{x,y \in \R_+,~c\in\mathbb{T}}=\frac{| \F(x) - c\F(y)|}{| x - c y|}.
\end{equation}

\begin{proposition} \label{lem:lip}
Suppose that $\F: \mathbb{R}_+ \to \C$ satisfies $\F(0) = 0$. Then
\begin{equation} \label{eq:rotlipM}
\|\F\|_{\Lip-\mathbb{C}} \leq \sqrt{2}\|\F\|_{\Lip}
\end{equation}
where $\sqrt{2}$ is optimal. If $\F$ is real-valued, $\F: \mathbb{R}_+ \to \R$, it holds that
\begin{equation} \label{eq:rotlipreal}
\|\F\|_{\Lip-\mathbb{C}} =\|\F\|_{\Lip}
\end{equation}
\end{proposition}
\begin{proof}
We may clearly assume that $\|\F\|_{\Lip}=1$. Suppose first that $\F$ is real-valued and write $c = a + ib$. Noting that the hypotheses imply that $|\F(x)| \leq x$ for all $x$, we have
$$ |\F(x)- c\F(y)|^2 = |\F(x) - \F(y)|^2 + 2(1-a) \F(x)\F(y) \leq |x-y|^2 + 2(1-a)xy = |x-cy|^2,$$ which shows that $\|\F\|_{\Lip-\mathbb{C}} \leq\|\F\|_{\Lip}$. Since the reverse inequality is obvious, \eqref{eq:rotlipreal} follows.

For $\F$ complex-valued, the computation is more involved. The inequality \eqref{eq:rotlipM} is clearly equivalent with
\begin{equation} \label{eq:rotlip}
|\F(x) - c\F(y)| \leq \sqrt{2}|x - cy|, \quad x,y \in \R_+, ~c\in\mathbb{T},
\end{equation}
(still assuming that $\|\F\|_{\Lip}=1$). Fix $c$ and suppose without loss of generality that $y \leq x$. If $\F(y) = 0$ there is nothing to prove since $|x-y| \leq |x-cy|$. Similarly, if $\F(x)=\F(y)$, \eqref{eq:rotlip} follows from the fact that $|\F(y)| \leq y$. If $\F(y) \neq 0$ and $\F(x) \neq \F(y)$, let $w \in \mathbb{C}$ be such that $\F(x) = w\F(y)$. Since $|\F(x)-\F(y)| = |w-1||\F(y)| \leq x-y$ there is a constant $d$, $0 < d \leq 1$ such that
$$|\F(y)| = d \frac{x-y}{|w-1|}.$$
Since $|\F(y)| \leq y$ it follows that
\begin{equation}\label{nana} x \leq \left(1 + \frac{|w-1|}{d}\right) y.\end{equation}
It is straightforward to check that the function $x \mapsto (x-y)/|x-cy|$ is increasing when $x \geq y$, which combined with \eqref{nana} yields that
$$\frac{x-y}{|x-cy|} \leq \frac{|w-1|/d}{ \left|1+\frac{|w-1|}{d}-c \right|} \leq \frac{|w-1|}{d|1+|w-1|-c|}.$$
Recalling that $\F(x) = w\F(y)$ we hence conclude that
$$|\F(x) - c\F(y)| = |w-c| d \frac{x-y}{|w-1|} \leq |w-c|\frac{|x-cy|}{|1+|w-1|-c|} .$$
We claim that
\begin{equation} \label{eq:maxnorm}
\sup_{w \in \C,\, |c|=1} \frac{|w-c|}{|1+|w-1|-c|} = \sqrt{2}.
\end{equation}
From \eqref{eq:maxnorm} we immediately deduce \eqref{eq:rotlip}, and it also implies that $\sqrt{2}$ is the best possible constant. To see the latter, fix $w \in \C$ and a unimodular $c$ and set $x=1+|w-1|$ and $y=1$. We may then define a function $\F$ as in the statement of the theorem satisfying that $\F(1) = 1$, $\F(1+|w-1|) = w$, which gives $$\frac{|\F(x)-c\F(y)|}{|x-cy|}=\frac{|w-c|}{|1+|w-1|-c|}.$$

To see that the supremum of \eqref{eq:maxnorm} is at least $\sqrt{2}$, let $w = 1 - it$ and let $c = e^{it}$. For this choice of $w$ and $c$ we have that
 $$\frac{|w-c|}{|1+|w-1|-c|} = \frac{|1-e^{it}-it|}{|1-e^{it} + t|} \to \sqrt{2}, \quad t \to 0.$$
 For the upper bound, note that upon squaring it is equivalent with
 $$ |w- c|^2 \leq 2|1+|w-1|-c|^2,$$
 which upon expanding, reordering and noting that $|c|=1$ is equivalent with
 $$|w|^2+5-4\re~w + 4|w-1| - 2\re \left[ \overline{c} (2(1+|w-1|)-w) \right] \geq 0. $$
 This inequality is true for all unimodular $c$ if and only if
  \begin{equation} \label{eq:diffeq}
  |w|^2+5-4\re ~w + 4|w-1| - 2|2(1+|w-1|)-w| \geq 0, \quad w \in \C.
  \end{equation}
  Since $|2(1+|w-1|)-w| \leq 2|w-1| + |w - 2|$ by the triangle inequality, \eqref{eq:diffeq} is implied by
  $$|w|^2+5-4\re~ w  - 2|w-2| \geq 0, \quad w \in \C. $$
 The left hand side equals
  $$|w-2|^2 + 1 - 2|w-2| =(|w-2|-1)^2,$$
which completes the proof.
\end{proof}

The next theorem is the key result of the paper. Theorem \ref{aoa} is an immediate corollary of this result and Proposition \ref{lem:lip}. Set \begin{equation}\label{SiOpLipS}\|\F_s\|_{\Lip}=\sup_{A,B\in\mathcal{S}_2}\frac{\|\F_s(A)-\F_s(B)\|_{\mathcal{S}_2}}{\|A-B\|_{\mathcal{S}_2}}.\end{equation}
It will also follow from the proof that the definition \eqref{SiOpLip} is independent of the dimension, as claimed in the introduction, and that it coincides with \eqref{SiOpLipS}.

\begin{theorem}\label{symmetricfcgeneral}
Let $\F:\mathbb{R}_+ \to \mathbb{C}$ be Lipschitz with $\F(0)=0$. Suppose that $A, B \in \mathcal{S}_2$. Then $\F_s$ satisfies
\begin{equation} \label{eq:oplipcomplex}
\|\F_s(A)-\F_s(B)\|_{\mathcal{S}_2} \leq \|\F\|_{\Lip-\mathbb{C}}\|A-B\|_{\mathcal{S}_2}
\end{equation}
and this estimate is optimal. In other words, $\|\F_s\|_{\Lip}=\|\F\|_{\Lip-\mathbb{C}}$.
\end{theorem}
\begin{proof}
We first prove \eqref{eq:oplipcomplex} for finite square matrices. That is, we will show that for $d \times d$ matrices $A$ and $B$ we have
\begin{equation} \label{eq:oplipfinite}
\|\F_s(A)-\F_s(B)\|_{F} \leq \|\F\|_{\Lip-\mathbb{C}}\|A-B\|_{F}.
\end{equation}
Suppose that this has been proved and consider general $A, B \in \mathcal{S}_2$ with singular value decompositions
$$A = \sum_{n=1}^\infty s_n(A) \, u^A_n \otimes v^A_n, \quad B = \sum_{n=1}^\infty s_n(A) \, u^B_n \otimes v^B_n.$$
For $N \geq 1$ we may consider
$$A_N = \sum_{n=1}^N s_n(A) \, u^A_n \otimes v^A_n, \quad B = \sum_{n=1}^N s_n(A) \, u^B_n \otimes v^B_n.$$
to be operators acting on $\mathcal{V}_N = \textrm{span} \, \{u^A_n, v^A_n, u^B_n, v^B_n \, | \, 1 \leq n \leq N\}$. Note that the singular value decompositions of $A_N$ and $B_N$ are identical whether considered operators on $\mathcal{H}$ or on $\mathcal{V}_N$, and that the exact same statement applies to $\F_s(A_N)$ and $\F_s(B_N)$. Hence \eqref{eq:oplipfinite} applied to the finite-dimensional space $\mathcal{V}_N$ gives us that
$$\|\F_s(A_N) - \F_s(B_N)|_{\mathcal{S}_2(\mathcal{H})} \leq \|\F\|_{\Lip-\mathbb{C}} \|A_N - B_N\|_{\mathcal{S}_2(\mathcal{H})}.$$
Since $\|A_N - A\|_{\mathcal{S}_2(\mathcal{H})} \to 0$ and $\|\F_s(A_N) - \F_s(A)\|_{\mathcal{S}_2(\mathcal{H})} \to 0$ as $N \to \infty$, and similarly for $B$, the inequality \eqref{eq:oplipcomplex} follows.

We now turn to proving \eqref{eq:oplipfinite} for $d \times d$-matrices $A$ and $B$, for which we express the singular value decompositions with the usual matrix notation;
$$A = U_A\Sigma_A V_A^*, \, \quad B = U_B\Sigma_B V_B^*.$$
We have the following formula for $\|A-B\|_F^2$,
 \begin{align*}
\|A-B\|_F^2&=\|U_A\Sigma_A V_A^*-U_B
\Sigma_B V_B^*\|_F^2=\|U_B^*U_A\Sigma_A -\Sigma_B
V_B^*V_A\|_F^2 \\ &=\|A\|^2_F+\|B\|_F^2-2\re \sum_{ae} \Sigma_B
\overline{V_B^*V_A} \odot U_B^*U_A\Sigma_A,\end{align*}
where $\sum_{ae}$ denotes the operation of summing all entries of a matrix,
and $\overline{M}$ denotes the action of taking the complex conjugate of every entry
of a matrix $M$.

Since $\overline{V_B^*V_A} \odot U_B^*U_A$ is cdss, Lemma \ref{lem:birk} implies that there exist
$c_1,\ldots,c_M$ in $[0,1]$ satisfying $\sum_{n=1}^Mc_n=1$, permutations $\pi_n$ and length-$d$ vectors $\gamma_n$ with unimodular entries, $1 \leq n \leq M$, such that
$$
\overline{V_B^*V_A} \odot U_B^*U_A=\sum_{n=1}^N c_n
M_{\pi_n,\gamma_n}.$$
We get
\begin{align*}\|A-B\|_F^2 &=\|A\|^2_F+\|B\|_F^2-2 \re \sum_{n=1}^Mc_n
\sum_{ae} \Sigma_B M_{\pi_n,\gamma_n}\Sigma_A
\\&=\sum_{n=1}^M c_n \left( \sum_{i=1}^Ns_i(A)^2+\sum_{i=1}^Ns_i(B)^2-2 \re \sum_{i=1}^N
s_i(B) \gamma_{n,i} s_{\pi_{n, i}}(A)\right)
\\&=\sum_{n=1}^M c_n \sum_{i=1}^N| s_i(B) - \gamma_{n,i} s_{\pi_{n,i}}(A)|^2.
\end{align*}
This identity applied to $\F_s(A) = U_A
\F(\Sigma_A) V_A^*$ and $\F_s(B) = U_B \F(\Sigma_B) V_B^*$ immediately gives
\begin{align*}&\|\F_s(A)-\F_s(B)\|_F^2=\sum_{n=1}^M c_n \sum_{i=1}^N| \F(s_i(B)) - \gamma_{n,i} \F(s_{\pi_{n,i}}(A))|^2
\end{align*}
Since $|\F(x)- \gamma \F(y)|\leq \|\F\|_{\Lip-\mathbb{C}}|x- \gamma y|$ for every $x, y \geq 0$ and $\gamma \in \mathbb{T}$, we get
\begin{align*}
&\|\F_s(A)-\F_s(B)\|_F^2 = \sum_{n=1}^M c_n \sum_{i=1}^N| \F(s_i(B)) - \gamma_{n,i} \F(s_{\pi_{n,i}}(A))|^2
\leq\\& \leq \|\F\|_{\Lip-\mathbb{C}}^2\sum_{n=1}^M c_n \sum_{i=1}^N| s_i(B) - \gamma_{n,i} s_{\pi_{n,i}}(A)|^2 = \|\F\|_{\Lip-\mathbb{C}}^2\|A-B\|_F^2,\end{align*}
which establishes \eqref{eq:oplipcomplex}. The optimality of \eqref{eq:oplipcomplex} follows immediately from the argument surrounding \eqref{optimal}. This also shows that $\|\F_s\|_{\Lip}=\|\F\|_{\Lip-\mathbb{C}}$ independent of whether we use definition \eqref{SiOpLipS} or \eqref{SiOpLip} with arbitrary fixed dimension.

\end{proof}

\bibliography{MCPerfekt}
\bibliographystyle{plain}

\end{document}